\documentclass[a4paper]{article}

\usepackage{amsmath}
\usepackage{amsthm}
\usepackage{amssymb}
\usepackage{amscd}

\begin{document}

\newtheorem{theorem}{Theorem}[section]
\newtheorem{lemma}[theorem]{Lemma}
\newtheorem{corollary}[theorem]{Corollary}

\renewcommand{\labelenumi}{(\roman{enumi})\quad}
\numberwithin{equation}{part}
\renewcommand{\theequation}{\fnsymbol{equation}}

\setlength{\unitlength}{1cm}

\newcommand{\scrH}{\mathcal{H}}
\newcommand{\F}{\mathcal{F}}

\newcommand{\qH}{\Gamma\backslash\scrH}
\newcommand{\qF}{\Gamma\backslash F}

\newcommand{\intF}{\stackrel{\circ}{F}}

\newcommand{\Id}{\mathrm{Id}}

\newcommand{\Nat}{\mathbb{N}}
\newcommand{\Zed}{\mathbb{Z}}
\newcommand{\Zplus}{\Zed^+}
\newcommand{\Comp}{\mathbb{C}}
\newcommand{\Real}{\mathbb{R}}
\newcommand{\Hyp}{\mathbb{H}^2}

\newcommand{\minus}{\!-\!}
\newcommand{\plus}{\!+\!}

\title{Regular tessellations of the hyperbolic plane by fundamental
  domains of a Fuchsian group}
\author{Robert Yuncken\\Penn State University}

\maketitle

\noindent Running Head: ``Regular tessellations and Fuchsian groups''


 generalizations.

\hspace{1.5cm}

\begin{abstract}

For positive integers $p$ and $q$ with $1/p+1/q<1/2$, a
tessellation of type $\{p,q\}$ is a tessellation of the hyperbolic
plane by regular $p$-gons with $q$ $p$-gons meeting at each vertex.
In this paper, a necessary and sufficient condition on 
the integers $p$ and $q$ is established 
to determine when a tessellation of type
$\{p,q\}$ can be realized as a tessellation of the hyperbolic plane by
fundamental domains of some Fuchsian group.  Specifically, a
tessellation of type $\{p,q\}$ is a tessellation by fundamental domains if and
only if $q$ has a prime divisor less than or equal to $p$.

\end{abstract}


It is well-known that, for integers $p$ and $q$ with $1/p+1/q<1/2$,
the hyperbolic plane $\Hyp$ can be tessellated by regular $p$-gons
with $q$ $p$-gons meeting at each vertex.  This is called a
tessellation of type $\{p,q\}$.  The group of all 
orientation-preserving symmetries of a
tessellation of type $\{p,q\}$ is a Fuchsian group, and a tessellation
of $\Hyp$ by fundamental domains for this group is obtained
by subdividing each of the regular $p$-gons in the original
tessellation into $p$ triangles sharing a vertex at its centre.

The question we deal with here is whether the $p$-gons themselves are
fundamental domains for some Fuchsian group.  Equivalently, one can
ask for which $p$ and $q$ does there exist a subgroup
of the full group of symmetries of the tessellation of type $\{p,q\}$
which acts freely and transitively
on the set of $p$-gons in the tessellation?
The following result answers this question.

\begin{theorem}
\label{thm}

The tessellation of type $\{p,q\}$ is a tessellation of the hyperbolic
plane by fundamental domains of some Fuchsian group if and only if $q$
has a prime divisor less than or equal to $p$.

\end{theorem}

The proof of this theorem can be broken into two Lemmas.

\begin{lemma}

The tessellation of type $\{p,q\}$ is a tessellation of the hyperbolic
plane by fundamental domains of some Fuchsian group if and only if
there exists an involution $\sigma$ in $S_p$ (the symmetric group on
$p$ letters) such that $(\sigma\rho)^q=1$, where $\rho$ is the cyclic
permutation $(1~2~\ldots~p)$ of order $p$.

\end{lemma}

\begin{proof}

Let us fix a $p$-gon $\F$ in the tessellation.  Let
$v_1, \ldots, v_p$ be the vertices of $\F$, labeled clockwise from some
arbitrary vertex $v_1$.  Let $e_1$ denote the edge between $v_p$ and
$v_1$, and let $e_i$ denote the edge between $v_{i-1}$ and
$v_i$ for $i=2,\ldots,p$.

First, let us prove the necessity of the combinatorial condition.
Suppose that $\F$ is a fundamental domain for some Fuchsian group
$\Gamma$.  For each of the $p$ $p$-gons adjacent to $\F$, there is a
unique element of $\Gamma$ which maps it onto $\F$.  This induces the
well-known edge-pairing for a (polygonal) fundamental domain---namely,
for each edge $e_i$ of $\F$ there is a unique non-identity element
$\gamma_i\in\Gamma$ and a unique edge of $\F$, which we will denote by
$e_{\sigma(i)}$, such that $\gamma_i$ maps $e_{\sigma(i)}$ onto $e_i$.
Note that we do not exclude the possibility that $e_i=e_{\sigma(i)}$,
in which case $\gamma_i$ is a rotation about the midpoint of $e_i$ by $\pi$.

Since $\gamma_i^{-1}$ maps $e_i$ onto $e_{\sigma(i)}$, we see that
$\gamma_i^{-1}=\gamma_{\sigma(i)}$ and $e_{\sigma(\sigma(i))}=e_i$.
Thus the edge-pairing of $\F$ provides an involution $\sigma\in S_p$.
We now prove the relation between $\sigma$ and $\rho$ by
considering the $q$ images of $\F$ about any vertex $v_i$.

Consider any edge $e_i$ of $\F$.  The adjacent edge of $\F$ which shares
the vertex $v_i$ is $e_{\rho(i)}$, where $\rho$ is the cyclic
permutation defined in the statement of the theorem.  By the
definition of $\sigma$, $\gamma_{\sigma\rho(i)}\F$ is the $p$-gon adjacent to
$\F$ in the tessellation which shares the edge $e_{\rho(i)}$,
and $\gamma_{\sigma\rho(i)}v_{\sigma\rho(i)} = v_i$.  The edge of
$\gamma_{\sigma\rho(i)}\F$ which is adjacent
to $\gamma_{\sigma\rho(i)}e_{\sigma\rho(i)}$ and
shares the vertex $v_i$ is
$\gamma_{\sigma\rho(i)}e_{\rho\sigma\rho(i)}$.  Noting that
$\gamma_{\sigma\rho(i)}\gamma_{\sigma\rho\sigma\rho(i)}\gamma_{\sigma\rho(i)}^{-1}$
maps the edge $\gamma_{\sigma\rho(i)}e_{\sigma\rho\sigma\rho(i)}$ of
$\gamma_{\sigma\rho(i)}\F$ to
$\gamma_{\sigma\rho(i)}e_{\rho\sigma\rho(i)}$, we see that $\gamma_{\sigma\rho(i)}\gamma_{\sigma\rho\sigma\rho(i)}\F$ is the $p$-gon
adjacent to $\gamma_{\sigma\rho(i)}\F$ sharing the edge
$\gamma_{\sigma\rho(i)}\gamma_{\sigma\rho\sigma\rho(i)}e_{\sigma\rho\sigma\rho(i)}
= \gamma_{\sigma\rho(i)}e_{\rho\sigma\rho(i)}$.
Continuing in this manner anti-clockwise about the vertex $v_i$, we obtain 
$$\gamma_{\sigma\rho(i)}\gamma_{\sigma\rho\sigma\rho(i)}\ldots
      \gamma_{(\sigma\rho)^q(i)}\F = \F$$ 
and
$$\gamma_{\sigma\rho(i)}\gamma_{\sigma\rho\sigma\rho(i)}\ldots
      \gamma_{(\sigma\rho)^q(i)}e_{(\sigma\rho)^q(i)}=e_i.$$
Since $\F$ was assumed to be a fundamental domain for $\Gamma$, the first
equality shows that 
$$\gamma_{\sigma\rho(i)}\gamma_{\sigma\rho\sigma\rho(i)}\ldots
      \gamma_{(\sigma\rho)^q(i)} = 1,$$
and thus the latter gives 
$$(\sigma\rho)^q(i)=i.$$
This holds for any $i=1,\ldots,p$, and hence $(\sigma\rho)^q=1$.

Conversely, suppose we have an involution $\sigma\in S_p$ with
$(\sigma\rho)^q=1$.  We generate the group $\Gamma$ by using the
edge-pairing isometries described by $\sigma$.  Specifically, let $\gamma_i$
be the unique orientation-preserving isometry of the hyperbolic plane
which maps the edge
$e_i$ onto $e_{\sigma(i)}$ (and does not map $\F$ onto itself).  Now let
$\Gamma$ be the subgroup of the full group of symmetries of the
tessellation of type $\{p,q\}$ which is generated by
$\gamma_1,\ldots,\gamma_p$.  Note that we immediately have the
relation $\gamma_i^{-1}=\gamma_{\sigma(i)}$.  We need to show that
$\Gamma$ acts freely and transitively on the set of $p$-gons in the
tessellation. 

For any $\gamma\in\Gamma$, the $p$-gon adjacent to $\gamma\F$ and
sharing the edge $\gamma e_i$ is $\gamma\gamma_{\sigma(i)}\F$.
Furthermore, the generator $\gamma_{\sigma(i)}$ which appears in this
context is clearly unique as such amongst the generators
$\gamma_1,\ldots,\gamma_p$.  This observation yields a one-to-one
correspondence between words in these generators (not using inverses)
and paths in the dual graph\footnote{The dual graph of the
  tessellation is the graph whose vertex set is the set of $p$-gons in
  the tessellation, with two $p$-gons being joined by an edge if and
  only if the $p$-gons are adjacent.}
of the tessellation beginning at $\F$.  The correspondence associates
the word $\gamma_{i_1}\ldots\gamma_{i_n}$ to the path passing
successively through the $p$-gons
$$\F,\, \gamma_{i_1}\F,\, \ldots,\, \gamma_{i_1}\!\ldots\!\gamma_{i_n}\F.$$
Furthermore, two such paths are homotopic in the dual graph (relative to their
endpoints) if and only if their corresponding words can be made equal
using only the relations $\gamma_i \gamma_{\sigma(i)} = 1$.

The connectedness of the dual graph now immediately implies the
transitivity of $\Gamma$ acting on the set of $p$-gons of the
tessellation.  It remains to demonstrate the freeness of this
action.  In the context of the dual graph, it needs to be shown that every
path beginning and ending at $\F$ corresponds to a word which
multiplies in $\Gamma$ to give the identity element.

Using the arguments from the first part of the proof, one can observe
that the hypothesis $(\sigma\rho)^q=1$ is precisely the condition that
the word corresponding to the path
around any given vertex $v_i$ of $\F$ multiplies in $\Gamma$ to give
the identity.  Let us call these ``primitive
paths''.  To complete the proof one only
needs to observe that any path in the dual graph which begins and ends
at $\F$ can be written as a concatenation of conjugates of these
primitive paths.

\end{proof}

Theorem \ref{thm} is now reduced to a combinatorial
problem.

\begin{lemma}

There exists an involution $\sigma\in S_p$ such that
$(\sigma\rho)^q=1$ if and only if $q$ has a divisor $d\neq1$ which is
less than or
equal to $p$.

\end{lemma}

\begin{proof}

If $(\sigma\rho)^q=1$ then $q$ is a multiple of the order of
$\sigma\rho$, which is greater than one and divides $p!$.  This proves
necessity.

For sufficiency, suppose $m\leq p$ is a divisor of $q$.  We will
explicitly construct an involution $\sigma\in S_p$ such that
$\sigma\rho$ has order $m$.    Put
$$p=am+r$$
for integers $a$ and $r$ with $0\leq r < m$.  Note that $a\geq1$ since
$m\leq p$.  Now define $\sigma$ as the following product of disjoint
transpositions:
\begin{eqnarray*}
 \sigma =
     \prod_{j=1}^{a-1} \Big(j(m\minus1)\plus1~~p\minus(j\minus1)\Big)
     \prod_{k=0}^{r-1} \Big(p\minus a\minus 2k~~p\minus a\minus (2k-1)\Big)
\end{eqnarray*}
This formula should be interpreted using the convention that if $a=1$
then the first product is empty, and if $r=0$ then the second product
is empty.
Calculation reveals that the product $\sigma\rho$ decomposes as the
following product of disjoint $m$-cycles:
\begin{eqnarray*}
\sigma\rho &=& \prod_{j=1}^{a-1} 
     \Big((j-1)(m\minus 1)\plus1~~(j-1)(m\minus 1)\plus
     2~~\ldots~~j(m\minus 1)~~p\minus (j\minus1)\Big) \cdot \\
&&     \Big((a\minus 1)(m\minus 1)\plus1~~(a\minus 1)(m\minus 1)\plus
     2~~\ldots \\
&& \hspace{1cm} \ldots~~
     p\minus a\minus 2r~~p\minus a\minus (2r\minus1)~~p\minus
        a\minus (2r\minus 3)~~\ldots\\
&& \hspace{6cm} \ldots~~
     p\minus a\minus 1~~p\minus a\plus1\Big).
\end{eqnarray*}

\end{proof}

The arguments above are not particular to hyperbolic
geometry.  The results hold equally in the Euclidean case
($1/p+1/q=1/2$) and the spherical case ($1/p+1/q>1/2$), where the
Fuchsian group is replaced by a discrete group of
orientation-preserving isometries of $\Real^2$ or $S^2$, respectively.
However, this does not add anything significant to the result since
there are only finitely many such regular tessellations.  In fact, one finds
that all such regular tessellations of the plane and sphere are
realisable as tessellations by fundamental domains of some discrete
group of orientation preserving isometries with one exception ---
the icosahedron
(type $\{3,5\}$).

\end{document}